\documentclass[14pt]{extarticle}
\usepackage{extsizes}
\usepackage{amssymb,amsfonts,amsthm}
\usepackage{amsmath}
\usepackage[makeroom]{cancel}
\usepackage{mathtools}
\usepackage{mathrsfs,pifont}
\usepackage{slashed,mathabx} 
\usepackage[bbgreekl]{mathbbol}
\usepackage{tikz}
\usepackage{framed}
\usetikzlibrary{calc}
\usepackage{mdframed}
\usepackage{amscd}

\usepackage[top=1in, bottom=1.25in, left=0.75in,right=0.75in]{geometry}
\usepackage{authblk}

\usepackage[font={small,sl}]{caption}
\usepackage[all]{xy}

\usepackage{hyperref}

\makeatletter
\newcommand*{\transpose}{%
  {\mathpalette\@transpose{}}%
}
\newcommand*{\@transpose}[2]{%
  \raisebox{\depth}{$\m@th#1\intercal$}%
}

\newtheorem{theorem}{Theorem}[section]
\newtheorem{lemma}[theorem]{Lemma}
\newtheorem{corollary}[theorem]{Corollary}
\theoremstyle{definition}

\newtheorem{proposition}[theorem]{Proposition}

\theoremstyle{remark}
\newtheorem{remark}[theorem]{Remark}
\numberwithin{equation}{section}

\makeatletter
\def\blfootnote{\xdef\@thefnmark{}\@footnotetext}
\makeatother

\begin{document}

\date{}
\title{}

\title{Aldous property for full-flag Johnson graphs}

\author{Gary Greaves\thanks{Division of Mathematical Sciences, Nanyang Technological University, 21 Nanyang Link, Singapore 637371. Email: gary@ntu.edu.sg.} \qquad Haoran Zhu\thanks{Division of Mathematical Sciences, Nanyang Technological University, 21 Nanyang Link, Singapore 637371. Email: zhuh0031@e.ntu.edu.sg} }

\maketitle

\begin{abstract}
We show that the full-flag Johnson graph has spectral gap equal to that of its Schreier quotient arising from the point-stabiliser equitable partition.
Our results confirm two conjectures posed by Huang, Huang, and Cioabă, which imply an
Aldous-type spectral-gap phenomenon for full-flag Johnson graphs. 
\end{abstract}

    

\section{Introduction}

The second-largest eigenvalue of a regular graph is widely regarded as a key indicator of its global connectivity and expansion properties~\cite{Alon,Alon2,Do,Mohar}. 
In particular, relative to the degree, a smaller second-largest eigenvalue typically reflects stronger connectivity, faster mixing of random walks, and better expansion.

The notion of the \textit{Aldous property} arises naturally at the interface of algebra, graph theory, and probability. 
A Cayley graph $\operatorname{Cay}(S_n,\Sigma)$ on the symmetric group $S_n$ with generating set $\Sigma$ is said to have the \textbf{Aldous property} if its second-largest eigenvalue coincides with that of the \textbf{Schreier graph} $\operatorname{Sch}(S_n,[n],\Sigma)$, that is, the graph with vertex set $[n]:=\{1,\dots,n\}$ and an arc from $i$ to $j$ for each $\sigma \in \Sigma$ satisfying $\sigma(i)=j$.

Aldous conjectured~\cite[Open Problem 14.29]{A} that $\operatorname{Cay}(S_n,\Sigma)$ has the Aldous property when $\Sigma$ is a set of transpositions. 
This conjecture stimulated nearly two decades of research, including partial results for random walks on finite groups~\cite{FOW}, Coxeter graphs~\cite{Bacher}, multipartite graphs~\cite{Cesi2}, and certain weighted graphs~\cite{Dieker}, before its full resolution by Caputo, Liggett, and Richthammer~\cite{CLR}.

The spectral theory of Cayley graphs can be distinguished by normality.
The Cayley graph $\operatorname{Cay}(G,\Sigma)$ is said to be \textbf{normal} if $\Sigma$ is closed under conjugation in $G$, that is, if $\Sigma$ is a union of conjugacy classes.
Normal Cayley graphs admit a particularly tractable spectral theory: their eigenvalues can be expressed explicitly in terms of the irreducible characters of the underlying group~\cite{DS,Z}.
In contrast, the spectral analysis of non-normal Cayley graphs is considerably more intricate. 
The absence of conjugacy invariance prevents a direct application of character theory, and comparatively few general results are currently available.

Variations and generalisations of Aldous' conjecture continue to attract attention~\cite{AKP,AP,Cesi1,CT,GZ,HH,HHC,LXZ1,LXZ2,PP}. 
Various families of Cayley graphs $\operatorname{Cay}(S_n,\Sigma)$ have been shown to possess the Aldous property for suitably structured generating sets $\Sigma$.
In the case of normal Cayley graphs, we refer the reader to \cite{HHC,PP}.
In the case of non-normal Cayley graphs, the Aldous property has been established when $\Sigma$ consists of reversals~\cite{CT}, prefix-reversals~\cite{Cesi1}, and certain subsets of cycles~\cite{LXZ2}.
As we show below (see Section~\ref{sec:cg}), in this paper, we are concerned with the (non-normal) case where $\Sigma$ consists of $(n-2)$-reducible permutations.

A \textbf{full flag} of $[n]$ is a chain $U=(U_1,\dots,U_n)$ of subsets satisfying $|U_i|=i$ and $U_i\subset U_{i+1}$. For integers $n$ and $k$ with $0<k<n$, the \textbf{Johnson graph} $J(n,k)$ is the graph whose vertices are the $k$-subsets of $[n]$, where two vertices $U$ and $V$ are adjacent whenever $|U\cap V|=k-1$.

In 2016, Dai~\cite{D1} introduced the \textbf{full-flag Johnson graph} ${{F\!\!J}}(n,k)$, for $0\le k<n$. Its vertex set consists of all full flags of $[n]$, and two full flags
\[
U=(U_1,\dots,U_n)
\quad \text{and} \quad
V=(V_1,\dots,V_n)
\]
are adjacent if and only if $\bigl|\{\, i\in[n] : U_i \neq V_i \,\}\bigr| = k$.
Equivalently, $U$ and $V$, viewed as collections of subsets of $[n]$, are adjacent if and only if they have exactly $n-k$ subsets in common; that is, $|U\cap V|=n-k$.

Dai~\cite{D1} observed that the permutahedron of order $n$ coincides with ${{F\!\!J}}(n,1)$, thereby placing ${{F\!\!J}}(n,k)$ as a natural generalisation of permutahedra. In~\cite{D2}, Dai calculated all eigenvalues of the quotient graph of ${{F\!\!J}}(n,1)$ with respect to a certain equitable partition, and conjectured that ${{F\!\!J}}(n,1)$ and its quotient graph share the same second-largest eigenvalue.

Huang, Huang, and Cioab\u{a}~\cite{HHC} observed that Dai's conjecture is a special case of Aldous' conjecture and initiated an analogous study of ${{F\!\!J}}(n,2)$ and implicitly conjectured that ${{F\!\!J}}(n,2)$ has the Aldous property~\cite[Conjecture 6.9]{LiThesis}.
Our main theorem resolves this conjecture.

\begin{theorem}
\label{thm:main}
For each $n\ge 4$, the full-flag Johnson graph ${{F\!\!J}}(n,2)$ has the Aldous property.
\end{theorem}

The statement of Theorem~\ref{thm:main} implies that ${{F\!\!J}}(n,2)$ is a Cayley graph on the symmetric group $S_n$, a fact already established by Dai~\cite{D1}.
However, ${{F\!\!J}}(n,2)$ is a non-normal Cayley graph, and hence character-theoretic methods are not directly applicable for deriving expressions for its eigenvalues~\cite{LXZ2}.

Our proof of Theorem~\ref{thm:main} follows an approach suggested by Huang, Huang, and Cioab\u{a}, which we outline in Section~\ref{sect:quotient-matrices}. To implement this strategy, we establish two conjectures of Huang, Huang, and Cioab\u{a}~\cite[Conjectures 23 and 24]{HHC} concerning recursive inequalities for the second-largest eigenvalue of quotient matrices arising from the point-stabiliser partition of the vertex set. The proof of these inequalities is primarily inductive and relies on several linear-algebraic arguments, viewed naturally from the perspective of Laplacian operators.

The paper is organised as follows.
In Section~\ref{sect:quotient-matrices}, we present the full-flag Johnson graph ${{F\!\!J}}(n,2)$ as a Cayley graph whose generating set consists of $(n-2)$-reducible permutations. 
We then reduce the proof of Theorem~\ref{thm:main} to establishing certain recursive inequalities (see Theorem~\ref{thm:HHC-ineq}).
In Section~\ref{sec:eigen-inequality}, we prove Theorem~\ref{thm:HHC-ineq}.
 
\section{Cayley graphs and equitable partitions}\label{sect:quotient-matrices}

\subsection{Full-flag Johnson graph as a Cayley graph}
\label{sec:cg}

Let $G$ be a finite group and let $\Sigma\subseteq G$ satisfying $1_G \notin \Sigma$ and $g^{-1} \in \Sigma$ for each $g \in \Sigma$.
The \textbf{Cayley graph} $\Gamma = \mathrm{Cay}(G,\Sigma)$ is the undirected graph whose vertex set is $G$ and two vertices $g,h\in G$ are adjacent whenever $h= \sigma g$ for some $\sigma\in \Sigma$.
The graph $\Gamma$ is clearly $|\Sigma|$-regular.
Furthermore, $\Gamma$ is connected if and only if $\Sigma$ generates $G$.

A permutation $\sigma\in S_n$ is called \textbf{$m$–reducible} if there exists a partition
$[n]=I_1\cup\cdots\cup I_m$ into contiguous intervals with $\sigma(I_j)=I_j$ for all $j$.
Further, we call $\sigma$ \textbf{maximally $m$-reducible} if $m$ is the largest possible integer such that $\sigma$ is $m$-reducible.
Define
\[
\mathcal R_n(k) :=\{\sigma\in S_n: \sigma\ \text{is maximally }(n-k)\text{–reducible}\}.
\]
Dai~\cite[Theorem 1]{D1} showed that ${{F\!\!J}}(n,k) \cong \operatorname{Cay}(S_n,\,\mathcal R_n(k))$.
The approach we take towards a proof of Theorem~\ref{thm:main} requires us to consider the subgraph of $\operatorname{Cay}\left (S_n,\mathcal R_n(2)\right )$ generated by those $(n-2)$–reducible permutations that move the symbol $1$, explicitly, 
\[
  \mathcal R_n^\prime (2):=\{(1,2,3), (1,3,2), (1,3), (1,2)(3,4), (1,2)(4,5),\ldots,(1,2)(n-1,n)\}.
\]
In fact, in addition to  $\operatorname{Cay}\left (S_n,\mathcal R_n(2)\right )$, we will also show that $\operatorname{Cay}\left (S_n,\mathcal R^\prime_n(2)\right )$ has the Aldous property (see Lemma~\ref{lem:n1}).

\subsection{Schreier graphs and quotient matrices}

Let $\Gamma=(V(\Gamma),E(\Gamma))$ be a simple undirected graph on $n$ vertices, and denote its adjacency matrix by $A_\Gamma$. 
A partition $\Pi = \{V_1, V_2, \dots, V_m\}$ of the vertex set $V(\Gamma)$ is called an \textbf{equitable partition} if, for every pair $i,j\in\{1,\dots ,m\}$, each vertex in $V_i$ has the same number $q_{ij}$ of neighbours in $V_j$.  
The matrix $Q_\Gamma(\Pi)=(q_{ij})_{m\times m}$ is called the \textbf{quotient matrix} of $\Gamma$ with respect to $\Pi$.

\begin{lemma}[{Godsil and Royle~\cite{GR}, Theorem 9.3.3}]
\label{lem:ep}
    Let $\Gamma$ be a graph and suppose that $\Pi$ is an equitable partition of $\Gamma$.
    Then $\det(xI-Q_\Gamma(\Pi))$ divides $\det(xI - A_\Gamma)$.
\end{lemma}

For $k \in [n]$, denote by $\operatorname{Stab}_n(k)$ the stabiliser subgroup of $S_n$ that fixes the element $k$.
Define $\Pi_n$ to be the set of left cosets of $\operatorname{Stab}_n(1)$.
Then each part of $\Pi_n$ is a left coset of the point-stabiliser of $S_n$, that is, $\Pi_n:=\{V_1,\dots,V_n\}$ where $V_i = (1i)\operatorname{Stab}_n(1)$ for each $i \in \{1,\dots,n\}$.
See Figure~\ref{fig:FJ(4,2)} a drawing of the graph ${{F\!\!J}}(4,2)$ with vertices grouped according to the point-stabiliser partition $\Pi_4$.

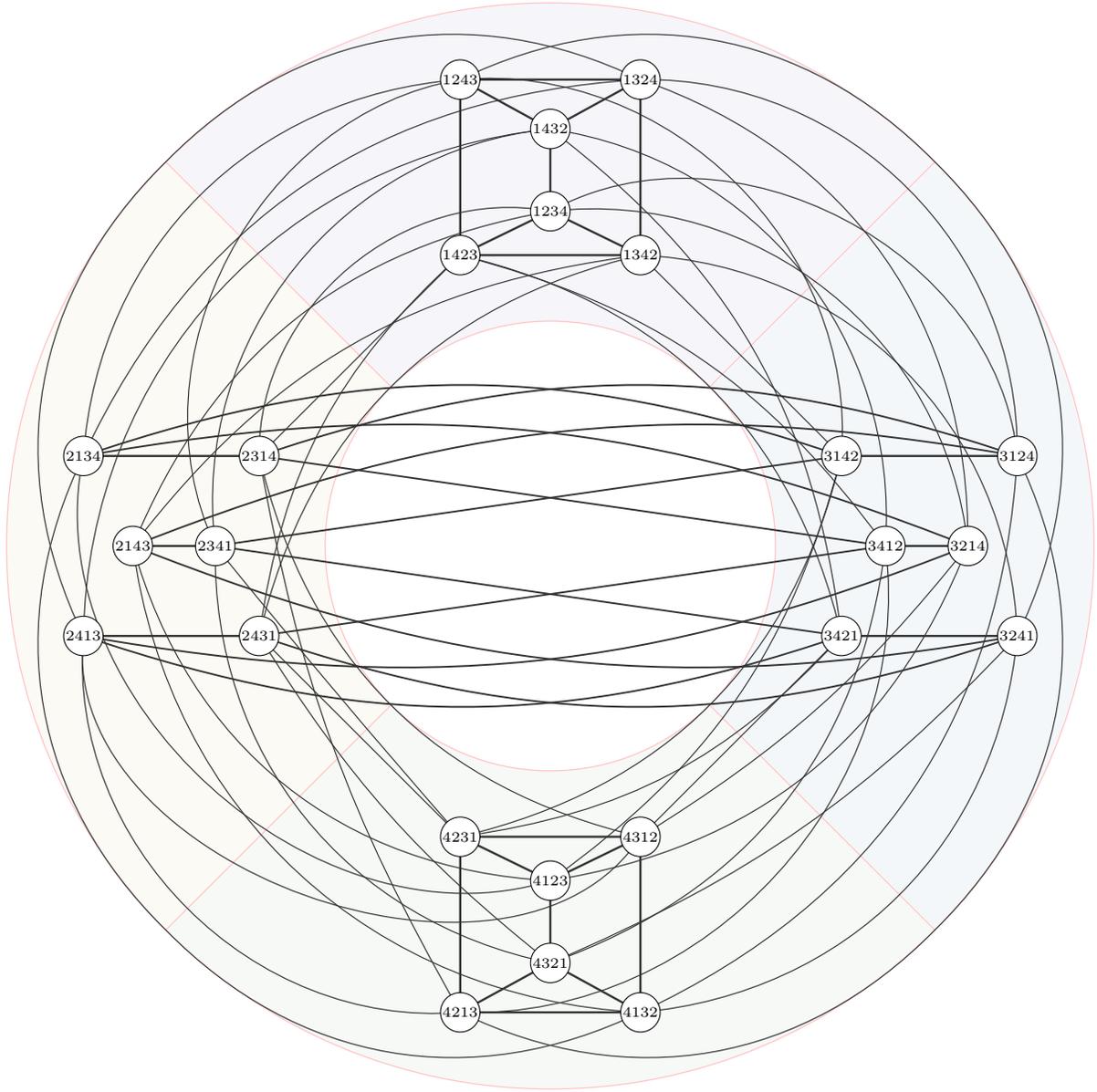
\begin{figure}[htbp]
\centering
\begin{tikzpicture}[line cap=round,line join=round,scale=1.6]

  \definecolor{fjviolet}{RGB}{186,176,214}
  \definecolor{fjblue}{RGB}{171,189,217}
  \definecolor{fjgreen}{RGB}{184,210,184}
  \definecolor{fjyellow}{RGB}{222,214,164}

  \tikzset{
    fjbg/.style={line width=0.55pt, draw=black!0, fill opacity=0.30},
    fjdivider/.style={draw=pink!90, line width=0.5pt},
    fjinner/.style={draw=black!80, line width=0.9pt},
    fjinter/.style={draw=black!80, line width=0.45pt},
    fjinter2/.style={draw=blue!90, line width=0.92pt},
    fjinterbundle/.style={draw=black!80, line width=0.72pt}
  }

  \def\Outer{4.95}   
  \def\Inner{2.05}

\path[fjbg,fill=fjviolet!40!white,even odd rule]
    (0,0) -- (45:\Outer) arc(45:135:\Outer) -- cycle
    (0,0) -- (45:\Inner) arc(45:135:\Inner) -- cycle;
  
\path[fjbg,fill=fjyellow!40!white,even odd rule]
    (0,0) -- (135:\Outer) arc(135:225:\Outer) -- cycle
    (0,0) -- (135:\Inner) arc(135:225:\Inner) -- cycle;

\path[fjbg,fill=fjgreen!40!white,even odd rule] 
    (0,0) -- (225:\Outer) arc(225:315:\Outer) -- cycle
    (0,0) -- (225:\Inner) arc(225:315:\Inner) -- cycle;

\path[fjbg,fill=fjblue!40!white,even odd rule]
    (0,0) -- (315:\Outer) arc(315:405:\Outer) -- cycle
    (0,0) -- (315:\Inner) arc(315:405:\Inner) -- cycle;

  \draw[fjdivider] (0,0) circle (\Outer);
  \draw[fjdivider] (0,0) circle (\Inner);

  \draw[fjdivider] (45:\Inner) -- (45:\Outer);
  \draw[fjdivider] (135:\Inner) -- (135:\Outer);
  \draw[fjdivider] (225:\Inner) -- (225:\Outer);
  \draw[fjdivider] (315:\Inner) -- (315:\Outer);

  \coordinate (v1_0) at (0.00,3.05);
  \coordinate (v1_2) at (-0.82,4.25);
  \coordinate (v1_1) at (0.82,4.25);
  \coordinate (v1_4) at (-0.82,2.65);
  \coordinate (v1_3) at (0.82,2.65);
  \coordinate (v1_5) at (0.00,3.80);

  \coordinate (v2_0) at (-3.55,0.85);
  \coordinate (v2_1) at (-4.25,0.43);
  \coordinate (v2_2) at (-2.85,0.43);
  \coordinate (v2_3) at (-4.25,-0.43);
  \coordinate (v2_4) at (-2.85,-0.43);
  \coordinate (v2_5) at (-3.55,-0.85);

\coordinate (v2_0) at (-3.80,0);
  \coordinate (v2_2) at (-4.25,0.82);
  \coordinate (v2_4) at (-2.65,0.82);
  \coordinate (v2_1) at (-4.25,-0.82);
  \coordinate (v2_3) at (-2.65,-0.82);
  \coordinate (v2_5) at (-3.05,0);

  \coordinate (v3_1) at (3.80,0);
  \coordinate (v3_2) at (2.65,0.82);
  \coordinate (v3_0) at (4.25,0.82);
  \coordinate (v3_5) at (2.65,-0.82);
  \coordinate (v3_3) at (4.25,-0.82);
  \coordinate (v3_4) at (3.05,0);

 \coordinate (v4_5) at (-0.82,-2.65);
  \coordinate (v4_4) at (-0.82,-4.25);
  \coordinate (v4_3) at (0.00,-3.80);
  \coordinate (v4_2) at (0.00,-3.05);
  \coordinate (v4_1) at (0.82,-2.65);
  \coordinate (v4_0) at (0.82,-4.25);

  \coordinate (v2_2-4_0) at (-3.5,-3.5);
    \coordinate (v1_2-3_3) at (3.5,3.5);
    \coordinate (v1_1-2_1) at (-3.5,3.5);
    \coordinate (v3_0-4_4) at (3.5,-3.5);

  \draw[fjinter] (v1_0) edge [bend right = 28.7]  (v2_0);
  \draw[fjinter] (v1_0) edge [bend right = 50.7]  (v2_4);
  \draw[fjinter] (v1_1) edge [bend right = 35.7]  (v1_1-2_1);
  \draw[fjinter] (v1_1-2_1) edge [bend right = 35.7]  (v2_1);
  \draw[fjinter] (v1_1) edge [bend right = 30.7]  (v2_2);
  \draw[fjinter] (v1_2) edge [bend right = 40.7]  (v2_2);
  \draw[fjinter] (v1_2) edge [bend right = 52.7]  (v2_5);
  \draw[fjinter] (v1_3) edge [bend right = 22.5]  (v2_0);
  \draw[fjinter] (v1_3) edge [bend right = 31.7]  (v2_3);
  \draw[fjinter] (v1_4) edge [bend right = 19.7] (v2_3);
  \draw[fjinter] (v1_4) edge [bend right = 0] (v2_4);
  \draw[fjinter] (v1_5) edge [bend right = 41.5] (v2_1);
  \draw[fjinter] (v1_5) edge [bend right = 45.7] (v2_5);

  \draw[fjinter] (v1_0) edge [bend left = 56.7] (v3_0);
  \draw[fjinter] (v1_0) edge [bend left =45.2] (v3_1);
  \draw[fjinter] (v1_1) edge [bend left = 45.7] (v3_0);
  \draw[fjinter] (v1_1) edge [bend left = 35.7] (v3_1);
  \draw[fjinter] (v1_2) edge [bend left = 49.7] (v3_2);
  \draw[fjinter] (v1_2) edge [bend left = 35.7] (v1_2-3_3);
  \draw[fjinter] (v1_2-3_3) edge [bend left = 35.7] (v3_3);
  \draw[fjinter] (v1_3) edge [bend left = 0] (v3_2);
  \draw[fjinter] (v1_3) edge [bend left = 42.7] (v3_3);
  \draw[fjinter] (v1_4) edge [bend left = 19.5] (v3_4);
  \draw[fjinter] (v1_4) edge [bend left = 31.7] (v3_5);
  \draw[fjinter] (v1_5) edge [bend left = 42.7] (v3_4);
  \draw[fjinter] (v1_5) edge [bend left = 22.7] (v3_5);

  \def\CL{-1.35}
  \def\CR{ 1.35}

  \draw[fjinterbundle] (v2_0) edge [bend left = 15.7] (v3_0);
  \draw[fjinterbundle] (v2_0)  edge [bend right = 15.7] (v3_3);
  \draw[fjinterbundle] (v2_1)  edge [bend right = 15.7] (v3_1);
  \draw[fjinterbundle] (v2_1)  edge [bend right = 18.7] (v3_5);
  \draw[fjinterbundle] (v2_2)  edge [bend left = 15.7] (v3_1);
  \draw[fjinterbundle] (v2_2) edge [bend left = 18.7] (v3_2);
  \draw[fjinterbundle] (v2_3) edge [bend right = 18.7] (v3_3);
  \draw[fjinterbundle] (v2_3) edge [bend left = 0] (v3_4);
  \draw[fjinterbundle] (v2_4) edge [bend left = 18.7] (v3_0);
  \draw[fjinterbundle] (v2_4) edge [bend right = 0] (v3_4);
  \draw[fjinterbundle] (v2_5) edge [bend right = 0] (v3_2);
  \draw[fjinterbundle] (v2_5) edge [bend left = 0] (v3_5);

  \draw[fjinter] (v2_0) edge [bend right = 39.7] (v4_0);
  \draw[fjinter] (v2_0) edge [bend right = 35.7] (v4_2);
  \draw[fjinter] (v2_1) edge [bend right = 76.7] (v4_1);
  \draw[fjinter] (v2_1) edge [bend right = 49.7] (v4_4);
  \draw[fjinter] (v2_2) edge [bend right = 35.7] (v2_2-4_0);
  \draw[fjinter] (v2_2-4_0) edge [bend right = 35.7] (v4_0);
  \draw[fjinter] (v2_2) edge [bend right = 59.7] (v4_2);
  \draw[fjinter] (v2_3) edge [bend right = 9.7] (v4_3);
  \draw[fjinter] (v2_3) edge [bend right = 0] (v4_5);
  \draw[fjinter] (v2_4) edge [bend right = 31.7] (v4_1);
  \draw[fjinter] (v2_4) edge [bend right = 9.7] (v4_4);
  \draw[fjinter] (v2_5) edge [bend right = 39.7] (v4_3);
  \draw[fjinter] (v2_5) edge [bend right = 1] (v4_5);

  \draw[fjinter] (v3_0) edge [bend left = 30.7] (v4_0);
  \draw[fjinter] (v3_0) edge [bend left = 35.7] (v3_0-4_4);
  \draw[fjinter] (v3_0-4_4) edge [bend left = 35.7] (v4_4);
  \draw[fjinter] (v3_1) edge [bend left = 9.7] (v4_1);
  \draw[fjinter] (v3_1) edge [bend left = 29.7] (v4_2);
  \draw[fjinter] (v3_2) edge [bend left = 19.7] (v4_2);
  \draw[fjinter] (v3_2) edge [bend left = 31.7] (v4_5);
  \draw[fjinter] (v3_3) edge [bend left = 39.7] (v4_0);
  \draw[fjinter] (v3_3) edge [bend left = 11.7] (v4_3);
  \draw[fjinter] (v3_4) edge [bend left = 31.7] (v4_3);
  \draw[fjinter] (v3_4) edge [bend left = 49.7] (v4_4);
  \draw[fjinter] (v3_5) edge [bend left = 0] (v4_1);
  \draw[fjinter] (v3_5) edge [bend left = 19.7] (v4_5);

  \draw[fjinner] (v1_0) -- (v1_3);
  \draw[fjinner] (v1_0) -- (v1_4);
  \draw[fjinner] (v1_0) -- (v1_5);
  \draw[fjinner] (v1_2) -- (v1_1);
  \draw[fjinner] (v1_2) -- (v1_4);
  \draw[fjinner] (v1_2) -- (v1_5);
  \draw[fjinner] (v1_1) -- (v1_3);
  \draw[fjinner] (v1_1) -- (v1_5);
  \draw[fjinner] (v1_3) -- (v1_4);

  \draw[fjinner] (v2_2) -- (v2_4);
  \draw[fjinner] (v2_0) -- (v2_5);
  \draw[fjinner] (v2_1) -- (v2_3);

  \draw[fjinner] (v3_0) -- (v3_2);
  \draw[fjinner] (v3_1) -- (v3_4);
  \draw[fjinner] (v3_3) -- (v3_5);

  \draw[fjinner] (v4_2) -- (v4_5);
  \draw[fjinner] (v4_2) -- (v4_1);
  \draw[fjinner] (v4_2) -- (v4_3);
  \draw[fjinner] (v4_0) -- (v4_4);
  \draw[fjinner] (v4_0) -- (v4_1);
  \draw[fjinner] (v4_0) -- (v4_3);
  \draw[fjinner] (v4_4) -- (v4_5);
  \draw[fjinner] (v4_4) -- (v4_3);
  \draw[fjinner] (v4_5) -- (v4_1);

  \node[font=\tiny, inner sep=0.5pt, circle,thin,draw=black!99,fill=white!100, minimum width=8pt]      at (v1_0) {1234};
  \node[font=\tiny, inner sep=0.5pt, circle,thin,draw=black!99,fill=white!100, minimum width=8pt] at (v1_1) {1324};
  \node[font=\tiny, inner sep=0.5pt, circle,thin,draw=black!99,fill=white!100, minimum width=8pt] at (v1_2) {1243};
  \node[font=\tiny, inner sep=0.5pt, circle,thin,draw=black!99,fill=white!100, minimum width=8pt] at (v1_3) {1342};
  \node[font=\tiny, inner sep=0.5pt, circle,thin,draw=black!99,fill=white!100, minimum width=8pt] at (v1_4) {1423};
  \node[font=\tiny, inner sep=0.5pt, circle,thin,draw=black!99,fill=white!100, minimum width=8pt]      at (v1_5) {1432};

  \node[font=\tiny, inner sep=0.5pt, circle,thin,draw=black!99,fill=white!100, minimum width=8pt]      at (v2_0) {2143};
  \node[font=\tiny, inner sep=0.5pt, circle,thin,draw=black!99,fill=white!100, minimum width=8pt] at (v2_1) {2413};
  \node[font=\tiny, inner sep=0.5pt, circle,thin,draw=black!99,fill=white!100, minimum width=8pt] at (v2_2) {2134};
  \node[font=\tiny, inner sep=0.5pt, circle,thin,draw=black!99,fill=white!100, minimum width=8pt] at (v2_3) {2431};
  \node[font=\tiny, inner sep=0.5pt, circle,thin,draw=black!99,fill=white!100, minimum width=8pt] at (v2_4) {2314};
  \node[font=\tiny, inner sep=0.5pt, circle,thin,draw=black!99,fill=white!100, minimum width=8pt]      at (v2_5) {2341};

  \node[font=\tiny, inner sep=0.5pt, circle,thin,draw=black!99,fill=white!100, minimum width=8pt]      at (v3_0) {3124};
  \node[font=\tiny, inner sep=0.5pt, circle,thin,draw=black!99,fill=white!100, minimum width=8pt] at (v3_1) {3214};
  \node[font=\tiny, inner sep=0.5pt, circle,thin,draw=black!99,fill=white!100, minimum width=8pt] at (v3_2) {3142};
  \node[font=\tiny, inner sep=0.5pt, circle,thin,draw=black!99,fill=white!100, minimum width=8pt] at (v3_3) {3241};
  \node[font=\tiny, inner sep=0.5pt, circle,thin,draw=black!99,fill=white!100, minimum width=8pt] at (v3_4) {3412};
  \node[font=\tiny, inner sep=0.5pt, circle,thin,draw=black!99,fill=white!100, minimum width=8pt]      at (v3_5) {3421};

  \node[font=\tiny, inner sep=0.5pt, circle,thin,draw=black!99,fill=white!100, minimum width=8pt]      at (v4_0) {4132};
  \node[font=\tiny, inner sep=0.5pt, circle,thin,draw=black!99,fill=white!100, minimum width=8pt] at (v4_1) {4312};
  \node[font=\tiny, inner sep=0.5pt, circle,thin,draw=black!99,fill=white!100, minimum width=8pt] at (v4_2) {4123};
  \node[font=\tiny, inner sep=0.5pt, circle,thin,draw=black!99,fill=white!100, minimum width=8pt] at (v4_3) {4321};
  \node[font=\tiny, inner sep=0.5pt, circle,thin,draw=black!99,fill=white!100, minimum width=8pt] at (v4_4) {4213};
  \node[font=\tiny, inner sep=0.5pt, circle,thin,draw=black!99,fill=white!100, minimum width=8pt]      at (v4_5) {4231};

\end{tikzpicture}
\caption{The graph $F\!\!J(4,2)$ with vertices labelled by the image of 1234 under the action of the corresponding permutation.
The vertices are grouped with respect to the point-stabiliser partition $\Pi_4$.}
\label{fig:FJ(4,2)}
\end{figure}

\begin{lemma}[{\cite[Lemma 5]{HHC}}]
\label{lem:cosetsEP}
    Let $G$ be a finite group.
Then the set of left cosets of any subgroup $H \le G$ gives an equitable partition of $\operatorname{Cay}(G, \Sigma )$.
\end{lemma}

The Schreier graph $\operatorname{Sch}(S_n, [n],\Sigma)$ is isomorphic to the graph with vertex set $\Pi_n = \{V_1,V_2,\dots,V_n\}$ where there is an arc from $V_i$ to $V_j$ for each $\sigma \in \Sigma$ and $y \in V_j$ satisfying $\sigma(y) = x$ for some fixed $x \in V_i$.

It follows from Lemma~\ref{lem:cosetsEP} that 
$\Pi_n$ is an equitable partition of both the Cayley graphs $\operatorname{Cay}(S_n,\mathcal R_n(2))$ and $\operatorname{Cay}(S_n,\mathcal R'_n(2))$.
Denote the corresponding quotient matrices by $\mathsf Q_n$ and $\mathsf Q^\prime_n$, respectively, i.e, 
\begin{align*}
    \mathsf Q_n &:= Q_{\operatorname{Cay}(S_n,\mathcal R_n(2))}(\Pi_n) \\
    \mathsf Q^\prime_n &:= Q_{\operatorname{Cay}(S_n,\mathcal R'_n(2))}(\Pi_n).
\end{align*}
It is straightforward to verify that the matrices $\mathsf Q_n$ and $\mathsf Q^\prime_n$ are adjacency matrices of the Schreier graphs $\operatorname{Sch}(S_n, [n],\mathcal R_n(2))$ and $\operatorname{Sch}(S_n, [n],\mathcal R'_n(2))$, respectively.
    
In \cite[Section 5]{HHC} the corresponding quotient matrices $\mathsf Q_n$ and $\mathsf Q_n^\prime$ are given as
\begin{equation*}%\label{eq:An}
\mathsf Q_n =  
\begin{bmatrix}
\frac{n^2-n-6}{2} & n-2 & 2 & 0 & \cdots & 0 & 0\\
n-2 & \frac{n^2-3n-2}{2} & n-2 & 2 & \ddots & \vdots & \vdots\\
2 & n-2 & \frac{n^2-3n-6}{2} & n-2 & \ddots & 0 & 0\\
0 & 2 & n-2 & \ddots & \ddots & 2& 0\\
\vdots & \ddots & \ddots & \ddots & \frac{n^2-3n-6}{2} & n-2 & 2\\
0 & \hdots & 0 & 2 & n-2 & \frac{n^2-3n-2}{2} & n-2\\
0 & \hdots & 0 & 0 & 2 & n-2 & \frac{n^2-n-6}{2}
\end{bmatrix}
\end{equation*}
and 
\[
\mathsf Q_n^\prime  = \begin{bmatrix}
0     & n-2 & 2   & 0   & \cdots & 0 & 0\\
n-2   & 1   & 1   & 0   & \cdots & 0 & 0\\
2     & 1   & n-4 & 1   & \ddots & 0 & 0 \\
0     & 0   & 1   & n-2 & \ddots & \ddots & \vdots \\
\vdots&\vdots&\ddots&\ddots&\ddots& 1 & 0 \\
0     & 0   & \ddots   & \ddots & 1    &  n-2 & 1 \\
0     & 0   & 0   & \cdots & 0 & 1    &  n-1
\end{bmatrix}.
\]

We will apply the following remark in Section~\ref{sec:eigen-inequality}.

\begin{remark}
    \label{rem:pf}
   Using the Gerschgorin theorem and the Perron-Frobenius theorem, it follows that both the matrices $\frac{n^2+n-6}{2}I - \mathsf Q_n$ and $nI - \mathsf Q_n^\prime$ are positive semidefinite with rank $n-1$.
\end{remark}

\subsection{Overview of the proof}
\label{sec:overview}

In this section, we outline and follow the proof strategy proposed by Huang, Huang, and Cioabă~\cite{HHC}.
Let $M$ be an $n \times n$ real symmetric matrix.
We write the eigenvalues of $M$ in non-increasing order as $\lambda_1(M)\geqslant\lambda_2(M)\geqslant\dots\geqslant\lambda_n(M)$.
To ease notation, for each $i \in [n]$, for a graph $\Gamma$, we define $\lambda_i(\Gamma) := \lambda_i(A_\Gamma)$.

Using the previous two subsections, in order to prove Theorem~\ref{thm:main}, it suffices to show that 
    \[
\lambda_2(\operatorname{Cay}(S_n,\mathcal R'_n(2))) = \lambda_2(\mathsf Q_{n}^\prime) \,\,\text{ and } \,\, \lambda_2(\operatorname{Cay}(S_n,\mathcal R_n(2))) = \lambda_2(\mathsf Q_{n}) .
    \]

We now record the following key theorem, which provides recursive inequalities for the second-largest eigenvalue of the quotient matrices $\mathsf Q_n$ and $\mathsf Q^\prime_n$.
We defer its proof until Section~\ref{sec:eigen-inequality}.

\begin{theorem}[{\cite[Inequality~21 and Inequality~22]{HHC}}]\label{thm:HHC-ineq}
Let $n\ge 5$. Then 
\begin{itemize}
    \item[(a)] $\lambda_2(\mathsf Q_n^\prime) \;>\; \lambda_2(\mathsf Q_{n-1}^\prime) + 1;$
    \item[(b)] $\lambda_2(\mathsf Q_n)\ >\ \lambda_2(\mathsf Q_{n-1}) + \lambda_2(\mathsf Q_n^\prime)$.
\end{itemize}
\end{theorem}

Huang, Huang, and Cioabă~\cite{HHC} proposed to use the following lemma, known as the `graph covering method', which bounds eigenvalues of a graph that are not eigenvalues of a quotient matrix of a certain equitable partition.

Given a square matrix $M$, we denote by $\Lambda(M)$ the set of its eigenvalues.

\begin{lemma}[{\cite[Theorem 7]{HHC}}]\label{lem:graph-covering-highly-trans}
Let $\Gamma = \operatorname{Cay}(S_n,\Sigma)$.
If $\lambda \in \Lambda(A_\Gamma) \setminus \Lambda(Q_\Gamma(\Pi_n))$ then,
for each $k\in[n]$, 
\[
\lambda\;\le\;
\lambda_{2}\bigl(\operatorname{Cay}(\operatorname{Stab}_n(k),\,\Sigma\cap \operatorname{Stab}_n(k))\bigr)\;+\;
\lambda_{2}\bigl(\operatorname{Cay}\bigl(S_n,\,\Sigma\setminus(\Sigma\cap \operatorname{Stab}_n(k))\bigr)\bigr).
\]
\end{lemma}

Using Lemma~\ref{lem:graph-covering-highly-trans} together with Theorem~\ref{thm:HHC-ineq}, we can show that $\operatorname{Cay}(S_n,\mathcal R'_n(2))$ has the Aldous property.

\begin{lemma}[{\cite[Conjecture 24]{HHC}}]
\label{lem:n1}
    Let $n \ge 4$.
    Then
    \[
\lambda_2(\operatorname{Cay}(S_n,\mathcal R'_n(2))) = \lambda_2(\mathsf Q_{n}^\prime).
    \]
\end{lemma}
\begin{proof}
    Let $\Gamma_n = \operatorname{Cay}(S_n,\mathcal R'_n(2))$.
We decompose the generating set $\mathcal{R}'_n(2)$ according to whether or not a generator
fixes $n$:
\begin{align*}
    \mathcal R'_n(2) &= (\mathcal R'_n(2) \cap \operatorname{Stab}_n(n) )\cup (\mathcal R'_n(2) \setminus \operatorname{Stab}_n(n)).
\end{align*}
Let $\lambda \in \Lambda(\Gamma_n) \setminus \Lambda(\mathsf Q_n^\prime)$.
Apply Lemma~\ref{lem:graph-covering-highly-trans} with
$\Sigma=\mathcal{R}'_n(2)$ and $k=n$ to obtain
% Apply Weyl's inequality to obtain
\begin{equation}\label{eq:nonquotient-FJ1}
  \lambda
  \ \le\
  \lambda_2\bigl(\operatorname{Cay}\left (\operatorname{Stab}_n(n),\mathcal R'_n(2) \cap \operatorname{Stab}_n(n)\right )\bigr) + \lambda_2\bigl(\operatorname{Cay}\left (S_n,\mathcal R'_n(2) \setminus \operatorname{Stab}_n(n) \right )\bigr).
\end{equation}

Observe that $\operatorname{Cay}\left (\operatorname{Stab}_n(n),\mathcal R'_n(2) \cap \operatorname{Stab}_n(n)\right ) = \operatorname{Cay}\left (S_{n-1},\mathcal R'_{n-1}(2) \right ) = \Gamma_{n-1}$ and $\operatorname{Cay}\left ( S_n,\mathcal R'_n(2) \setminus \operatorname{Stab}_n(n) \right ) = \frac{n!}{2} K_2$.
Hence, \eqref{eq:nonquotient-FJ1} becomes

\begin{equation}\label{eq:nonquotient-FJ12}
  \lambda
  \ \le\
  \lambda_2\bigl(\Gamma_{n-1}\bigr) + 1.
\end{equation}

Now we assume that $\lambda_2\bigl(\Gamma_{n-1}\bigr) = \lambda_2\bigl(\mathsf Q_{n-1}^\prime\bigr)$.
After applying Theorem~\ref{thm:HHC-ineq} (a), \eqref{eq:nonquotient-FJ12} becomes 
\begin{equation}\label{eq:nonquotient-FJ123}
  \lambda
  \ <\
  \lambda_2\bigl(\mathsf Q_{n}^\prime\bigr).
\end{equation}
Combining with Lemma~\ref{lem:ep}, we can deduce $\lambda_2(\Gamma_n)
  =
  \lambda_2\bigl(\mathsf Q_{n}^\prime\bigr)$ by induction, checking the base case $n = 4$ directly.
\end{proof}

Similarly, armed with Lemma~\ref{lem:n1} together with Lemma~\ref{lem:graph-covering-highly-trans} and Theorem~\ref{thm:HHC-ineq}, we can now state and prove the following lemma, which is equivalent to Theorem~\ref{thm:main}.

\begin{lemma}[{\cite[Conjecture 23]{HHC}}]
\label{lem:n}
    Let $n \ge 4$.
    Then
    \[
    \lambda_2(\operatorname{Cay}(S_n,\mathcal R_n(2))) = \lambda_2(\mathsf Q_{n}).
    \]
\end{lemma}
\begin{proof}
   Let $\Gamma_n = \operatorname{Cay}(S_n,\mathcal R_n(2))$.
We decompose the generating set $\mathcal{R}_n(2)$ according to whether or not a generator
fixes $1$:
\[
  \mathcal{R}_n(2)
  = \mathcal{R}'_n(2) \cup \bigl(\mathcal{R}_n(2)\setminus \mathcal{R}'_n(2)\bigr) = \mathcal{R}'_n(2) \cup \bigl(\mathcal{R}_n(2)\cap \operatorname{Stab}_n(1)\bigr),
\]

Let $\lambda \in \Lambda(\Gamma_n) \setminus \Lambda(\mathsf Q_n)$.
Apply Lemma~\ref{lem:graph-covering-highly-trans} with $\Sigma=\mathcal{R}_n(2)$, and $k=1$,
we obtain 
\begin{equation}\label{eq:nonquotient-FJ}
  \lambda
  \ \le\
  \lambda_2\bigl(\operatorname{Cay}(\operatorname{Stab}_n(1),\mathcal{R}_n(2)\cap \operatorname{Stab}_n(1))\bigr)
  + \lambda_2\bigl(\operatorname{Cay}(S_n,\mathcal{R}'_n(2))\bigr).
\end{equation}

Since $\operatorname{Cay}(\operatorname{Stab}_n(1),\mathcal{R}_n(2)\cap\operatorname{Stab}_n(1)) = \operatorname{Cay}(S_{n-1},\mathcal{R}_{n-1}(2))$, \eqref{eq:nonquotient-FJ} becomes
\begin{equation}
    \label{eq:44}
      \lambda
  \ \le\
  \lambda_2\bigl(\operatorname{Cay}(S_{n-1},\mathcal{R}_{n-1}(2))\bigr) + \lambda_2\bigl(\operatorname{Cay}(S_n,\mathcal R'_n(2))\bigr).
\end{equation}

Assume for some $n\ge 5$, the induction hypothesis
\[
  \lambda_2\bigl(\operatorname{Cay}(S_{n-1},\mathcal{R}_{n-1}(2))\bigr) = \lambda_2\bigl(\mathsf Q_{n-1}\bigr).
\]
Together with Lemma~\ref{lem:n1}, \eqref{eq:44} becomes
\[
  \lambda
  \ \le\
  \lambda_2\bigl(\mathsf Q_{n-1}\bigr) + \lambda_2\bigl(\mathsf Q_n^\prime\bigr).
\]
Now apply Theorem~\ref{thm:HHC-ineq} (b) to obtain $\lambda < \lambda_2\bigl(\mathsf Q_n\bigr)$.

Now Lemma~\ref{lem:ep} yields $\lambda_2\bigl(\operatorname{Cay}(S_n,\mathcal R_n(2))\bigr) = \lambda_2\bigl(\mathsf Q_n\bigr)$ by induction, verifying the base case $n = 4$ directly.
\end{proof}

\section{Two eigenvalue inequalities}\label{sec:eigen-inequality}

In this section, we prove Theorem~\ref{thm:HHC-ineq}, thus completing the proof of Theorem~\ref{thm:main}.
We treat each of the two inequalities of Theorem~\ref{thm:HHC-ineq} separately.
For notational convenience, we denote the eigenvalues of an $n \times n$ real symmetric matrix $M$ in non-decreasing order as $\mu_1(M)\leqslant\mu_2(M)\leqslant\cdots\leqslant\mu_n(M)$.

Define $\mathbf 1_n \in \mathbb R^n$ to be the all-ones vector, that is, each entry of $\mathbf 1_n$ is equal to $1$.
Furthermore, we define the Laplacian operator $\mathfrak L(M)$ by $\mathfrak L(M):=\operatorname{diag}(M\mathbf 1_n)-M$.
We note here that $\mathfrak L(M) \mathbf 1_n = \mathbf 0_n$, the zero vector.

\subsection{The first inequality}

In this section, we prove Theorem~\ref{thm:HHC-ineq} (a), whose statement is equivalent to that of Proposition~\ref{prop:eigen-Ln1-drop}, below.
Note that
\[
\mathfrak L(\mathsf Q_n^\prime)  = \begin{bmatrix}
n     & 2-n & -2   & 0   & \cdots & 0 & 0\\
2-n   & n-1   & -1   & 0   & \cdots & 0 & 0\\
-2     & -1   & 4 & -1   & \ddots & 0 & 0 \\
0     & 0   & -1   & 2 & \ddots & \ddots & \vdots \\
\vdots&\vdots&\ddots&\ddots&\ddots& -1 & 0 \\
0     & 0   & \ddots   & \ddots & -1    &  2 & -1 \\
0     & 0   & 0   & \cdots & 0 & -1    &  1
\end{bmatrix}.
\]
We begin with a rough bound on the second-smallest eigenvalue of $\mathfrak L(\mathsf Q_n^\prime)$.

\begin{lemma}\label{lem:upperbound for mu2(Ln)}
Let $n \ge 6$.
Then
\[
\mu_2(\mathfrak L(\mathsf Q_n^\prime))\leqslant 2-2\cos\Bigl(\dfrac{2\pi}{n-3}\Bigr).
\]
In particular, for $n\ge 9$ one has $\mu_2(\mathfrak L(\mathsf Q_n^\prime))\leqslant 1$.
\end{lemma}
\begin{proof}
By deleting the rows and columns of $\mathfrak L(\mathsf Q_n^\prime)$ indexed by elements of $\{1,2,3,n\}$, one obtains $2I - A_{P_{n-4}}$, where $P_{n-4}$ is the path graph on $n-4$ vertices.

Using eigenvalue interlacing \cite[Corollary 1]{F} and \cite[Section 1.4.4]{spectra}, yields
\[
\mu_2(\mathfrak L(\mathsf Q_n^\prime)) \le  \mu_2(2I - A_{P_{n-4}}) \ =\ 2-2\cos\left(\frac{2\pi}{n-3}\right ).   \qedhere
\]
\end{proof}

Next, we obtain an upper bound for the difference of the first two entries of an $\mu_2(\mathfrak L(\mathsf Q_n^\prime))$-eigenvector of $\mathfrak L(\mathsf Q_n^\prime)$.

\begin{lemma}\label{lem:bound-v1v2}
Let $n\ge 9$ and $\mathbf v=(v_1,\dots,v_n)^\transpose$ be a $\mu_2(\mathfrak L(\mathsf Q_n^\prime))$-eigenvector of $\mathfrak L(\mathsf Q_n^\prime)$.
Then
\[
  |v_2-v_1|\ \le\ \frac{2\mu_2\!\left(\mathfrak L(\mathsf Q_n^\prime)\right)\left |\cos\left (\tfrac{(2n-5)\theta}{2}\right )\right |}{n\cos(\tfrac{\theta}{2})}\, |v_n|,
\]
for some $\theta\in (0,\pi/3]$.
\end{lemma}

\begin{proof}
Write $\mu=\mu_2(\mathfrak L(\mathsf Q_n^\prime))$.
Taking the first two rows of the eigenvalue equation $\mathfrak L(\mathsf Q_n^\prime)\mathbf v=\mu\,\mathbf v$ gives
\begin{equation}\label{eq:v1v2}
\begin{cases}
(n-\mu)\,v_1 + (2-n)\,v_2 &= 2\,v_3,\\[2pt]
(2-n)\,v_1 + (n-1-\mu)\,v_2 &= v_3.
\end{cases}
\end{equation}
This is a $2\times2$ linear system in $(v_1,v_2)$ with coefficient matrix
\[
M=
\begin{bmatrix}
n-\mu & 2-n\\
2-n   & n-1-\mu
\end{bmatrix}.
\]
Its determinant is
\[
\det M
=(n-\mu)(n-1-\mu)-(n-2)^2,
\]
which is nonzero by Lemma~\ref{lem:upperbound for mu2(Ln)}. Hence, the system \eqref{eq:v1v2} has a unique solution.

Solving \eqref{eq:v1v2} yields
\[
v_1
=
\frac{2(n-1-\mu)-(2-n)}{(n-\mu)(n-1-\mu)-(n-2)^2}\,v_3
=
\frac{3n-4-2\mu}{(n-\mu)(n-1-\mu)-(n-2)^2}\,v_3,
\]
\[
v_2
=
\frac{(n-\mu)-2(2-n)}{(n-\mu)(n-1-\mu)-(n-2)^2}\,v_3
=
\frac{3n-4-\mu}{(n-\mu)(n-1-\mu)-(n-2)^2}\,v_3.
\]
Therefore,
\[
    v_2-v_1
  =\frac{\mu}{(n-\mu)(n-1-\mu)-(n-2)^2}\,v_3.
\]
For $n\ge 9$, using Lemma~\ref{lem:upperbound for mu2(Ln)}, one finds that
\[
  (n-\mu)(n-1-\mu)-(n-2)^2 \;>\; \frac{n}{2}.
\]
Whence,
\begin{equation}
\label{eqn:ineq1}
    |v_2-v_1|
  =\frac{\mu}{(n-\mu)(n-1-\mu)-(n-2)^2}\,|v_3|
  \;\le\; \frac{2}{n}\,\mu\,|v_3|.
\end{equation}
From the eigenvalue equation $\mathfrak L(\mathsf Q_n^\prime)\mathbf v=\mu\,\mathbf v$, the rows with indices
$i=4,\dots,n-1$ give
\begin{equation}
    \label{eqn:rr}
     v_{i+1} = (2-\mu)v_i - v_{i-1}.
\end{equation}
The last row yields the boundary condition
\begin{equation}
    \label{eqn:bc}
     v_{n-1} = (1-\mu)v_n.
\end{equation}

Since $n\ge 9$, Lemma~\ref{lem:upperbound for mu2(Ln)} yields the inequality $0<\mu\le 1$.
Hence, the characteristic equation $r^2-(2-\mu)r+1=0$ of \eqref{eqn:rr} has 
roots $\exp(\pm \sqrt{-1}\theta)$ for some $\theta\in(0,\pi/3]$ that satisfies $2\cos\theta=2-\mu$.
Furthermore, for each $i\ge 3$, we can write
\[
v_i = c_1\cos((n-i)\theta)+c_2\sin((n-i)\theta),
\]
with $c_1= v_n$ and $c_1\cos(\theta)+c_2\sin(\theta)=(1-\mu)v_n$ from \eqref{eqn:bc}.
Therefore,
\begin{equation}
    \label{eqn:v31}
    v_3 = \cos((n-3)\theta)v_n+(\cos \theta-1)\frac{\sin((n-3)\theta)}{\sin \theta}v_n.
\end{equation}
Using standard trigonometric identities, \eqref{eqn:v31} can be expressed as
\[
v_3
= \frac{\cos\bigl(\tfrac{(2n-5)\theta}{2}\bigr)}{\cos(\tfrac{\theta}{2})}v_n.
\]
Combining with \eqref{eqn:ineq1} completes the proof. \qedhere
\end{proof}

Now we are ready to state and prove the main result of this subsection, which is equivalent to Theorem~\ref{thm:HHC-ineq} (a).

\begin{proposition}\label{prop:eigen-Ln1-drop}
Let $n\geqslant 4$. Then
\[
\mu_2(\mathfrak L(\mathsf Q_{n+1}^\prime))<\mu_2(\mathfrak L(\mathsf Q_{n}^\prime)).
\]  
\end{proposition}

\begin{proof}
Let $\mathbf v=(v_1,\dots,v_n)^\transpose$ be an eigenvector of $\mathfrak L(\mathsf Q_n^\prime)$
corresponding to the second-smallest eigenvalue $\mu=\mu_2(\mathfrak L(\mathsf Q_n^\prime))$.
Let $\mathbf x = (v_1,\dots,v_n,v_n)^\transpose -\frac{v_n}{n+1}\,\mathbf 1_{n+1}$.
    Since $\mathbf v$ is orthogonal to $\mathbf 1_n$, we also have $\mathbf x^\transpose \mathbf 1_{n+1} = 0$.
    Hence, using Remark~\ref{rem:pf}, we deduce $\mu_2(\mathfrak L(\mathsf Q_{n+1}^\prime)) \leqslant \frac{\mathbf x^\transpose \mathfrak L(\mathsf Q_{n+1}^\prime) \mathbf x}{\mathbf x^\transpose \mathbf x}$.
    We can express $\mathfrak L(\mathsf Q_{n+1}^\prime)$ as 
    \[
    \mathfrak L(\mathsf Q_{n+1}^\prime)
  =
  \begin{pmatrix}
    \mathfrak L(\mathsf Q_{n}^\prime) & 0\\ 0 & 0
  \end{pmatrix}
  +(\mathbf e_1-\mathbf e_2)(\mathbf e_1-\mathbf e_2)^\transpose
  +(\mathbf e_n-\mathbf e_{n+1})(\mathbf e_n-\mathbf e_{n+1})^\transpose,
    \]
    where $\mathbf e_i$ denotes the standard basis vector in $\mathbb R^{n+1}$ whose $i$th entry is $1$ and all other entries equal $0$.
    It is straightforward to check that 
    \[
    \frac{\mathbf x^\transpose \mathfrak L(\mathsf Q_{n+1}^\prime) \mathbf x}{\mathbf x^\transpose \mathbf x} = \frac{\mu_2(\mathfrak L(\mathsf Q_{n}^\prime))\,\mathbf v^\transpose \mathbf v + (v_1-v_2)^2}
         {\mathbf v^\transpose \mathbf v + \dfrac{n}{n+1}\,v_n^2},
    \]
It now suffices to prove
\begin{equation}\label{eq:target-ineq}
  (v_1-v_2)^2 \;\le\; \frac{n}{n+1}\,\mu\,v_n^2.
\end{equation}

Assume that $n \ge 9$.
By Lemma~\ref{lem:bound-v1v2} we have
\[
  (v_1-v_2)^2
  \ \le\ \frac{4}{n^2}\,\mu^2\,
        \frac{\cos^2\bigl(\tfrac{2n-5}{2}\theta\bigr)}{\cos^2(\tfrac{\theta}{2})}\,v_n^2.
\]
By Lemma~\ref{lem:upperbound for mu2(Ln)}, we have $0\le\mu\le 1$.
Thus $0<\theta\le \pi/3$ and hence
\[
  (v_1-v_2)^2
  \ \le\ \frac{4}{n^2}\,\mu^2\cdot \frac{1}{(\sqrt{3}/2)^2}\,v_n^2
  = \frac{16}{3n^2}\,\mu^2\,v_n^2,
\]
which is less than $\frac{n}{n+1}\,\mu\,v_n^2$ when $n \ge 9$.

The remaining cases $4\leq n\leq 8$ can be verified directly by computer.
\end{proof}

\subsection{The second inequality}

In this section, we prove Theorem~\ref{thm:HHC-ineq} (b).
Note that 
\begin{equation*}%\label{eq:An}
\mathfrak L(\mathsf Q_n) =  
\begin{bmatrix}
n & 2-n & -2 & 0 & \cdots & 0 & 0\\
2-n & 2n-2 & 2-n & -2 & \ddots & \vdots & \vdots\\
-2 & 2-n & 2n & 2-n & \ddots & 0 & 0\\
0 & -2 & 2-n & \ddots & \ddots & -2& 0\\
\vdots & \ddots & \ddots & \ddots & 2n & 2-n & -2\\
0 & \hdots & 0 & -2 & 2-n & 2n-2 & 2-n\\
0 & \hdots & 0 & 0 & -2 & 2-n & n
\end{bmatrix}.
\end{equation*}
% We spell out the details of this equivalence at the end of the section.
We begin, as in the previous section, with a rough bound on the second-smallest eigenvalue of $\mathfrak L(\mathsf Q_n)$.

\begin{lemma}\label{lem:upperbound-mu2-Ln}
Let $n\ge 4$. Then
\begin{equation}\label{eq:main}
  \mu_2(\mathfrak L(\mathsf Q_n))\ \le\
  \frac{12\,(n-2)(n+7)}{n(n^2-1)}.
\end{equation}
In particular, $\mu_2(\mathfrak L(\mathsf Q_n))\le 4$ for all $n\ge 5$.
\end{lemma}

\begin{proof}
Let $\mathbf x = (x_1,\dots,x_n) \in \mathbb R^n$ be defined by 
\[
  x_i := i-\frac{n+1}{2},\qquad 1\le i\le n.
\]
Then the vector $\mathbf x$ is orthogonal to $\mathbf 1_n$ and hence
\[
  \mu_2(\mathfrak L(\mathsf Q_n))\ \le\ \frac{\mathbf x^\transpose \mathfrak L(\mathsf Q_n) \mathbf x}{\mathbf x^\transpose \mathbf x}.
\]

Using the pentadiagonal structure of  $\mathfrak L(\mathsf Q_n)$ one checks that, for any $\mathbf y\in\mathbb R^n$,
\[
  \mathbf y^\transpose \mathfrak L(\mathsf Q_n) \mathbf y
  = (n-2)\sum_{i=1}^{n-1}(y_{i+1}-y_i)^2
    + 2\sum_{i=1}^{n-2}(y_{i+2}-y_i)^2.
\]
For the specific choice of $\mathbf x$ above, we have
  $x_{i+1}-x_i = 1$ and $x_{i+2}-x_i = 2$, whence,
\[
  \mathbf x^\transpose \mathfrak L(\mathsf Q_n) \mathbf x
  = (n-2)(n-1) + 2\cdot 4(n-2)
  = (n-2)(n+7).
\]

It follows that
\[
  \mu_2(\mathfrak L(\mathsf Q_n))
  \le \frac{\mathbf x^\transpose \mathfrak L(\mathsf Q_n)\mathbf x}{\mathbf x^\transpose  \mathbf x}
  = \frac{12\,(n-2)(n+7)}{n(n^2-1)},
\]
as required.
\end{proof}

Let $A$ be a real symmetric matrix.
Denote by $A[i,j]$ the $(i,j)$-entry of $A$.
We call $A$ a 
\textbf{Robinson matrix}
if
\[
\begin{cases}
A[i,j]\le A[i,k]\qquad\text{whenever } j<k<i,
\\
A[i,j]\ge A[i,k]\qquad\text{whenever } i<j<k.
\end{cases}
\]
Equivalently, for every row/column, the off-diagonal entries of $A$ are nonincreasing as the distance to the diagonal increases.

The matrix $A$ is called \textbf{irreducible} if it cannot be decomposed into a nontrivial direct sum of matrices.
We call a vector $\mathbf x = (x_1,\dots,x_n)$ \textbf{non-increasing} if $x_i \le x_{i+1}$ for each $i \in \{1,\dots,n-1\}$.
An eigenvector for the second-smallest eigenvalue of $\mathfrak L(A)$ is called a \textbf{Fiedler vector}.

\begin{lemma}[{\cite[Theorem~3.2 and Theorem~4.6]{ABH}}]\label{lem:ABH-monotone}
Let $A$ be a 
Robinson matrix. 
Then
\begin{enumerate}
\item $\mathfrak L(A)$ has a non-increasing Fiedler vector.
\item If $A$ is irreducible with $\min_{i,j} A[i,j]=0$, then $\mu_2(\mathfrak L(A))>0$ and $\mu_2(\mathfrak L(A))$ is simple.
\end{enumerate}
\end{lemma}

For each $n\in\mathbb N$, the \textbf{exchange matrix} $J_n$ is the permutation
matrix of order $n$ with ones on the anti-diagonal and zeros elsewhere.
The exchange matrix satisfies $J_n^2=I$ and $J^{-1}_n=J_n$. 
A real matrix $A$ of order $n$ is called \textbf{centrosymmetric} if $AJ_n = J_nA$ and a vector $\mathbf x=(x_1,\dots,x_n)^\transpose\in\mathbb C^n$ is called
\textbf{skew-symmetric} if
$J_n\mathbf x=-\mathbf x$ (that is, $x_i=-x_{n+1-i}$ for all $i$). 

\begin{lemma}[{\cite[Theorem~2(i)]{Andrew}}]\label{lem:centro}
Let $A$ be a centrosymmetric matrix of order $n$.
Then, for each eigenspace of $A$ admits a basis consisting
entirely of vectors $\mathbf x$ satisfying $J_n\mathbf x=\mathbf x$ or $J_n\mathbf x=-\mathbf x$.
\end{lemma}

Clearly, $\mathsf Q_n$ is an irreducible, symmetric Robinson matrix.
Furthermore, $\mathfrak L(\mathsf Q_n)$ is centrosymmetric for $n \ge 4$.
Thus, applying Lemma~\ref{lem:ABH-monotone} and Lemma~\ref{lem:centro} yields the following corollary.

\begin{corollary}\label{cor:fiedler-monotoneskewsym}
Let $n\ge 4$.
Then $\mathfrak L(\mathsf Q_{n})$ has a Fiedler vector that is non-increasing and skew-symmetric and $\mu_2(\mathfrak L(\mathsf Q_{n}))$ is a simple eigenvalue.
\end{corollary}

We are now ready to prove the main result of this section, from which the proof of Theorem~\ref{thm:HHC-ineq} (b) follows.

\begin{proposition}\label{prop:eigen-Ln-drop}
Let $n\geqslant 4$. Then
\[
\mu_2(\mathfrak L(\mathsf Q_{n+1}))<\mu_2(\mathfrak L( \mathsf Q_{n})).
\]  
\end{proposition}
\begin{proof}

Let $\mathbf v = (v_1,\dots,v_n)^\transpose$ be a unit eigenvector of $\mathfrak L(\mathsf Q_n)$ corresponding to the second-smallest eigenvalue $\mu=\mu_2(\mathfrak L(\mathsf Q_n))$.
% $\mu_2(\mathfrak L( \mathsf Q_{n}))$-eigenvector for $\mathfrak L( \mathsf Q_{n})$.  
Let $\mathbf x = (v_1,\dots,v_n,v_n)^\transpose -\frac{v_n}{n+1}\,\mathbf 1_{n+1}$.
    Since $\mathbf v$ is orthogonal to $\mathbf 1_n$, we also have $\mathbf x^\transpose \mathbf 1_{n+1} = 0$.
    Hence, using Remark~\ref{rem:pf}, we deduce $\mu_2(\mathfrak L( \mathsf Q_{n+1})) \leqslant \frac{\mathbf x^\transpose \mathfrak L( \mathsf Q_{n+1})\mathbf x}{\mathbf x^\transpose \mathbf x}$.
    We can express $\mathfrak L( \mathsf Q_{n+1})$ as
\begin{align*}
  \mathfrak L( \mathsf Q_{n+1})
  &= \begin{pmatrix}\mathfrak L( \mathsf Q_{n})&0\\0&0\end{pmatrix}
   + \sum_{i=1}^{n-1}(\mathbf e_i-\mathbf e_{i+1})(\mathbf e_i-\mathbf e_{i+1})^{\transpose} \\
  &\quad + (n-1)(\mathbf e_n-\mathbf e_{n+1})(\mathbf e_n-\mathbf e_{n+1})^{\transpose}
        + 2(\mathbf e_{n-1}-\mathbf e_{n+1})(\mathbf e_{n-1}-\mathbf e_{n+1})^{\transpose},
\end{align*}
where $\mathbf e_i$ denotes the standard basis vector in $\mathbb R^{n+1}$ whose $i$th entry is $1$ and all other entries equal $0$.
    It is straightforward to check that 
    \[
    \frac{\mathbf x^\transpose \mathfrak L( \mathsf Q_{n+1}) \mathbf x}{\mathbf x^\transpose \mathbf x} = \frac{\mu + \sum_{i=1}^{n-1}(v_i-v_{i+1})^2+2(v_{n-1}-v_n)^2}
         {1 + \dfrac{n}{n+1}\,v_n^2},
    \]
    whence,
    \[
    \mu_2(\mathfrak L( \mathsf Q_{n+1}))\le \frac{\mu + \sum_{i=1}^{n-1}(v_i-v_{i+1})^2+2(v_{n-1}-v_n)^2}
         {1 + \dfrac{n}{n+1}\,v_n^2}.
    \]
For $n\ge 15$, it therefore remains to prove
\begin{equation}\label{eq:target-ineq2}
  \sum_{i=1}^{n-1}(v_i-v_{i+1})^2+2(v_{n-1}-v_n)^2 \;<\; \frac{n}{n+1}\,\mu\,v_n^2.
\end{equation}

Set $d_i:=v_{i+1}-v_i$ for $1\le i\le n-1$.
From the pentadiagonal form of $\mathfrak L(\mathsf Q_n)$ it follows (as in the proof of Lemma~\ref{lem:upperbound-mu2-Ln}) that
\[
  \mu = \mathbf v^\transpose \mathfrak L(\mathsf Q_n) \mathbf v
  = (n-2)\sum_{i=1}^{n-1} d_i^2
    + 2\sum_{i=1}^{n-2}(d_i+d_{i+1})^2.
\]
Expanding the second sum yields

\begin{align*}
  \mu
  &= (n-2)\sum_{i=1}^{n-1} d_i^2
     + 2\left(2\sum_{i=1}^{n-1} d_i^2 - (d_1^2+d_{n-1}^2)
       + 2\sum_{i=1}^{n-2} d_i d_{i+1}\right)\\
  &= (n+2)\sum_{i=1}^{n-1} d_i^2
     - 2(d_1^2+d_{n-1}^2)
     + 4\sum_{i=1}^{n-2} d_i d_{i+1}.
\end{align*}

By Corollary~\ref{cor:fiedler-monotoneskewsym}, the Fiedler vector $\mathbf v$ can be chosen to be
non-increasing and skew-symmetric. 
Therefore $d_i\ge 0$ for all $i$, and hence
\begin{equation}
\label{ineq:1}
      \mu
  \ \ge\
  (n+2)\sum_{i=1}^{n-1} d_i^2 - 2(d_1^2+d_{n-1}^2).
\end{equation}

Skew-symmetry gives $v_1=-v_n$ and $v_2=-v_{n-1}$, so
\[
  d_1 = v_2-v_1 = -v_{n-1}+v_n = v_n-v_{n-1} = d_{n-1},
\]
and hence $d_1^2=d_{n-1}^2=(v_n-v_{n-1})^2$.
Substituting into \eqref{ineq:1} yields
\[
  \sum_{i=1}^{n-1}(v_{i+1}-v_i)^2
  = \sum_{i=1}^{n-1} d_i^2
  \ \le\
  \frac{\mu+4(v_n-v_{n-1})^2}{n+2}.
\]
We rewrite this inequality as 
\begin{equation}
\label{ineq:2}
  \sum_{i=1}^{n-1}(v_i-v_{i+1})^2+2(v_{n-1}-v_n)^2
\le \frac{\mu}{n+2}
     + 2\!\left(1+\frac{2}{n+2}\right)\!(v_{n-1}-v_n)^2.
\end{equation}

From the last row of the eigenvalue equation $\mathfrak L(\mathsf Q_n)\mathbf v=\mu \mathbf v$ we obtain
\[
  (n-2)(v_n-v_{n-1}) + 2(v_n-v_{n-2}) = \mu v_n.
\]
Since $\mathbf v$ is non-increasing (i.e., $v_{n-2}\le v_{n}$) we obtain $(n-2)(v_n - v_{n-1}) \;\le\; \mu v_n$.

Substituting this into \eqref{ineq:2} yields
\[
  \sum_{i=1}^{n-1}(v_i-v_{i+1})^2+2(v_{n-1}-v_n)^2
  \;\le\;
  \frac{\mu}{n+2}
  +\frac{2(n+4)\mu^2}{(n+2)(n-2)^2}\,v_n^2,
\]
Comparing the right-hand side of the above to that of \eqref{eq:target-ineq2}, it thus suffices to show
\[
 v_n^2\left (\frac{n(n+2)}{n+1}-\frac{2(n+4)\mu}{(n-2)^2} \right ) > 1.
\]
Since $\mathbf v$ is a non-increasing unit vector, $v_n^2 \ge 1/n$, moreover,
\[
v_n^2\left (\frac{n(n+2)}{n+1}-\frac{2(n+4)\mu}{(n-2)^2} \right ) \ge \frac{(n+2)}{n+1}-\frac{2(n+4)\mu}{n(n-2)^2}.
\]
By Lemma~\ref{lem:upperbound-mu2-Ln} we have $\mu\le 4$ for all $n\ge 5$.
Hence,
\[
 \frac{(n+2)}{n+1}-\frac{2(n+4)\mu}{n(n-2)^2} \ge \frac{(n+2)}{n+1}-\frac{8(n+4)}{n(n-2)^2},
\]
which is more than $1$ for $n \ge 15$.

For the remaining cases $4\le n\le 14$, one can directly verify the strict inequality
$\mu_2(\mathfrak L(\mathsf Q_{n+1}))<\mu_2(\mathfrak L(\mathsf Q_n))$ using a computer.
\end{proof}

Finally, we can complete the proof of Theorem~\ref{thm:HHC-ineq}.

\begin{proof}[Proof of Theorem~\ref{thm:HHC-ineq}(b)]
First, observe that
\[
  \lambda_2(\mathsf Q_n)
  \;=\; \frac{n^2+n-6}{2} - \mu_2(\mathfrak L(\mathsf Q_n)),
\]
and 
\[
  \lambda_2(\mathsf Q_{n-1})
  \;=\;  \frac{n^2-n-6}{2} - \mu_2(\mathfrak L(\mathsf Q_{n-1})).
\]

Clearly, the Schreier graph with adjacency matrix $\mathsf Q_n^\prime$ is $n$-regular and connected.
Hence,
\[
  \lambda_1(\mathsf Q_n^\prime) = n
  \quad\text{and}\quad
  \lambda_2(\mathsf Q_n^\prime) < n.
\]
By Proposition~\ref{prop:eigen-Ln-drop}, we have
$\mu_2(\mathfrak L(\mathsf Q_{n-1}))-\mu_2(\mathfrak  L(\mathsf Q_n))>0$ and thus
\[
  \lambda_2(\mathsf Q_n)
  - \bigl(\lambda_2(\mathsf Q_{n-1}) + \lambda_2(\mathsf Q_n^\prime)\bigr)
  \;=\; \bigl(\mu_2(\mathfrak L(\mathsf Q_{n-1}))-\mu_2(\mathfrak L(\mathsf Q_n))\bigr)
       + \bigl(n-\lambda_2(\mathsf Q_n^\prime)\bigr)
  \;>\; 0,
\]
as required.
\end{proof}

\bigskip
\noindent{\bf Acknowledgements.} GG was supported by the Singapore Ministry of Education Academic Research Fund; grant numbers: RG14/24 (Tier 1) and MOE-T2EP20222-0005 (Tier 2). The authors are grateful to Yuxuan Li for bringing this problem to their attention.

\bibliographystyle{plain}

\begin{thebibliography}{99}



\bibitem{A}
D.~Aldous and J.A.~Fill,
\emph{Reversible Markov Chains and Random Walks on Graphs},
unfinished monograph, 2002 (recompiled version, 2014), available online at
\url{http://www.stat.berkeley.edu/~aldous/RWG/book.html} (last checked: 25/02/2026).


%\href{https://www.stat.berkeley.edu/users/aldous/Research/OP/sgap.html}{Spectral Gap for the Interchange (Exclusion) Process on a Finite Graph (last checked:25/02/2026)}.

\bibitem{AKP} 
G. Alon, G. Kozma, and D. Puder,
\textit{On the Aldous-Caputo spectral gap conjecture for hypergraphs},
Math.\ Proc.\ Cambridge Philos.\ Soc. \textbf{179} (2025) 259-298.


\bibitem{Alon}
N.~Alon, 
\textit{Eigenvalues and expanders}, 
Combinatorica \textbf{6} (1986), no.~2, 83--96.

\bibitem{Alon2}
N.~Alon and V.D.~Milman, 
\textit{$\lambda_1$, isoperimetric inequalities for graphs, and superconcentrators}, 
J. Combin. Theory Ser. B, \textbf{38}(1) (1985), 73--88.

\bibitem{AP}
G.~Alon and D.~Puder,
\textit{Aldous-type spectral gaps in unitary groups},
preprint,  \href{https://arXiv.org/abs/2603.00353}{arXiv:2603.00353}, (2026).
%, doi:10.48550/arXiv.2603.00353.

\bibitem{Andrew}
A.L.~Andrew, 
\textit{Eigenvectors of certain matrices},
Linear Algebra Appl. \textbf{7} (1973), no.~2, 151--162.

\bibitem{ABH}
J.E. Atkins, E.G. Boman, and B. Hendrickson,
\textit{A spectral algorithm for seriation and the consecutive ones problem},
SIAM J. Comput. \textbf{28} (1998), no. 1, 297–310.

\bibitem{Bacher}
R.~Bacher,
\textit{Valeur propre minimale du laplacien de Coxeter pour le groupe symétrique}, 
J. Algebra \textbf{167}(2) (1994), 460--472.

% \bibitem{BCIM}
% A. E. Brouwer, S. M. Cioab\u{a}, F. Ihringer, and M. McGinnis,
% \textit{The smallest eigenvalues of Hamming graphs, Johnson graphs and other distance-regular graphs with classical parameters},
% J. Combin. Theory, Ser. B, \textbf{133} (2018), 88–121.

\bibitem{spectra}
A.E. Brouwer and W.H. Haemers, Spectra of graphs, Springer, New York, 2012.

\bibitem{CLR}
P. Caputo, T. Liggett, and T. Richthammer,
\textit{Proof of Aldous’ spectral gap conjecture},
J. Amer. Math. Soc., \textbf{23}(3) (2010), 831–851.

\bibitem {Cesi1} 
F. Cesi, 
\textit{Cayley graphs on the symmetric group generated by initial reversals have unit spectral gap},
Electron.\ J.\ Combin.\ \textbf{16} (2009), no.~1, \#N29.

\bibitem{Cesi2}
F. Cesi,
\textit{On the eigenvalues of Cayley graphs on the symmetric group generated by a complete multipartite set of transpositions},
J.\ Algebraic Combin.\ \textbf{32} (2010), no.~2, 155--185.

% \bibitem{CGK}
% X.-M. Cheng, G.R.W. Greaves, and J.H. Koolen,
% \textit{Graphs with three eigenvalues and second largest eigenvalue at most~1},
% J. Combin. Theory Ser. B \textbf{129} (2018), 55--78.

% \bibitem{CGGK}
% X.-M. Cheng, A.L.~Gavrilyuk, G.R.W.~Greaves, and J.H. Koolen,
% \textit{Biregular graphs with three eigenvalues},
% European J.\ Combin.\ \textbf{56} (2016), 57--80.


\bibitem{CT}
F. Chung and J. Tobin, 
\textit{The spectral gap of graphs arising from substring reversals},
Electron.\ J.\ Combin.\ \textbf{24} (2017), no.~3, \#P3.4.

% \bibitem{CY}
% F. Chung, S.-T. Yau,
% \textit{Coverings, heat kernels and spanning trees},
% Electron. J. Combin., \textbf{6} (1999), \#R12.

\bibitem{D1}
I.  Dai,
\textit{Combinatorial properties of Full-Flag Johnson graphs}, 
In \textit{Combinatorial Algorithms},  vol. 9538 of \textit{Lecture Notes in Comput. Sci.}, pages 112–123, Springer,2016.

\bibitem{D2}
I. Dai,
\textit{Diameter bounds and recursive properties of Full-Flag Johnson graphs},
Discrete Math., \textbf{341}(7) (2018), 1932–1944.

\bibitem{DS}
P.~Diaconis and M.~Shahshahani,
\textit{Generating a random permutation with random transpositions},
Z.\ Wahrscheinlichkeitstheorie verw.\ Gebiete \textbf{57} (1981), no.~2, 159--179.

\bibitem{Dieker}
A.B. Dieker,
\textit{Interlacings for random walks on weighted graphs and the interchange process}, 
SIAM J. Discrete Math. \textbf{24}(1) (2010), 191--206.

\bibitem{Do}
J.~Dodziuk, 
\textit{Difference equations, isoperimetric inequality and transience of certain random walks}, 
Trans. Amer. Math. Soc., \textbf{284}(2) (1984), 787--794.

\bibitem{F}
S. Fisk, 
\textit{A very short proof of Cauchy’s interlace theorem for eigenvalues of Hermitian matrices},
Amer. Math. Monthly, \textbf{112} (2005), no.~2, 118--118.

\bibitem{FOW}
L.~Flatto, A.M. Odlyzko, and D.B. Wales,
\textit{Random shuffles and group representations}, 
Ann. Probab. \textbf{13}(1) (1985), 154--178.

\bibitem{GZ} 
G. Greaves and H. Zhu, 
\textit{A note on some spectral properties of generalised pancake graphs}, \href{https://arXiv.org/abs/2509.09425}{arXiv:2509.09425}, (2025), to appear in Discrete Mathematics.

\bibitem{GR}
C. Godsil, G. Royle,
\textit{Algebraic Graph Theory}, volume 207 of \textit{Graduate Texts in Mathematics}. Springer-Verlag, New York, 2001.

\bibitem{HH} 
X. Huang and Q. Huang, 
\textit{The second largest eigenvalues of some Cayley graphs on alternating groups}, 
J. Algebraic Combin. \textbf{50} (2019), no.~1, 99--111.

\bibitem{HHC}
X. Huang, Q. Huang, and S. Cioab\u{a},
\textit{The second eigenvalue of some normal Cayley graphs of highly transitive groups},
Electron. J. Combin. \textbf{26}, (2019), no. 2, \#P2.44.

\bibitem{LiThesis}
Y.~Li,
\textit{\href{https://findanexpert.unimelb.edu.au/scholarlywork/1962159-the-second-largest-eigenvalue-of-cayley-graphs-on-symmetric-groups}{The second largest eigenvalue of Cayley graphs on symmetric groups}}, 
Ph.D.\ thesis, The University of Melbourne, 2024.

\bibitem{LXZ1}
Y. Li, B. Xia and S. Zhou, 
\textit{The second largest eigenvalue of normal Cayley graphs on symmetric groups generated by cycles}, 
J. Combin. Theory Ser. A, \textbf{206} (2024) 105885.

\bibitem{LXZ2}
Y. Li, B. Xia and S. Zhou, 
\textit{The second largest eigenvalue of some nonnormal Cayley graphs on symmetric groups}, 
J. Combin. Theory Ser. A, \textbf{218} (2026) 106097.

\bibitem{Mohar}
B.~Mohar, 
\textit{Isoperimetric numbers of graphs}, 
J. Combin. Theory Ser. B, \textbf{47}(3) (1989), 274--291.

% \bibitem{NPR}
% S. Noschese, L. Pasquini, and L. Reichel,
% \textit{Tridiagonal Toeplitz matrices: properties and novel applications},
% Numer. Linear Algebra Appl. \textbf{20} (2013), 302–326.

\bibitem{PP}
O. Parzanchevski and D. Puder,
\textit{Aldous’s spectral gap conjecture for normal sets},
Trans. Amer. Math. Soc., \textbf{373}(10) (2020), 7067–7086.

% \bibitem{Sagan}
% B. E. Sagan,
% \textit{The symmetric group: Representations, combinatorial algorithms, and symmetric functions},
% 2nd ed., Graduate Texts in Mathematics, vol.~203, Springer-Verlag, New York, 2001.

% \bibitem{Serre}
% J.-P. Serre,
% \textit{Linear Representations of Finite Groups},
% Graduate Texts in Mathematics, vol.~42, Springer-Verlag, 1977.

\bibitem{Z}
P.H. Zieschang, \textit{Cayley graphs of finite groups}, J. Algebra \textbf{118} (1988), no.~2, 447--454.

\end{thebibliography}

\end{document}